\documentclass[11pt,a4paper]{article}
\usepackage{times}
\usepackage{graphicx}
\usepackage{amsfonts}
\usepackage{mathtools} 
\usepackage{esint} 
\usepackage{amsthm}
\usepackage{tikz}
\usetikzlibrary{shapes,arrows,positioning}
\usepackage{amssymb,epsfig,epstopdf,titling,url,array}
\usepackage{textcomp}
\usepackage{multirow}
\usepackage[colorlinks=true,linkcolor=blue, citecolor=blue]{hyperref}%
\usepackage[square,numbers]{natbib}

\theoremstyle{plain}
\newtheorem{thm}{Theorem}[section]

\newtheorem{prop}[thm]{Proposition}
\newtheorem{cor}[thm]{Corollary}

\theoremstyle{definition}

\theoremstyle{remark}

\makeatletter
\newbox\xrat@below
\newbox\xrat@above
\newcommand{\xrightarrowtail}[2][]{%
  \setbox\xrat@below=\hbox{\ensuremath{\scriptstyle #1}}%
  \setbox\xrat@above=\hbox{\ensuremath{\scriptstyle #2}}%
  \pgfmathsetlengthmacro{\xrat@len}{max(\wd\xrat@below,\wd\xrat@above)+.6em}%
  \mathrel{\tikz [>->,baseline=-.75ex]
                 \draw (0,0) -- node[below=-2pt] {\box\xrat@below}
                                node[above=-2pt] {\box\xrat@above}
                       (\xrat@len,0) ;}}
\makeatother

\begin{document}
\pagestyle{plain}

\begin{center}
{\fontsize{14}{20}\bf
A note on the rational homotopy type of the projectivization of the tangent bundle of complex projective spaces.
}
\end{center}

\begin{center}
   \textbf{Meshach Ndlovu }\\
Department of Mathematics and Statistical Sciences,\\ Faculty of Science,\\ Botswana International University of Science and Technology,\\ e-mail: mathswithmesh@gmail.com

 \textbf{Jean Baptiste Gatsinzi}\\
Department of Mathematics and Statistical Sciences,\\ Faculty of Science,\\ Botswana International University of Science and Technology,\\ e-mail: gatsinzij@biust.ac.bw
\end{center}

\section*{Abstract}
In this paper, we determine the rational homotopy type of the total space of the projectivization of the complex tangent bundle $\tau : \mathbb{C}^n \longrightarrow E \longrightarrow \mathbb{C}P^{n}$. We show that the total space $P(E)$ of the projectivization bundle $P(\tau):  \mathbb{C}P^{n-1} \longrightarrow P(E) \longrightarrow \mathbb{C}P^{n}$ has the rational homotopy type of $\mathrm{U}(n+1)/\mathrm{U}(1)\times \mathrm{U}(1)\times \mathrm{U}(n-1)$. 

\vspace{3mm}
\textit{Mathematics Subject Classification} (2020): Primary 55P62; Secondary 55P15. 

\vspace{3mm}
\textbf{Keywords:} Projectivization bundle; Sullivan models; Complex projective space; Homogeneous space.

\section{Introduction}
The projectivization of a complex vector bundle $\xi :  \mathbb{C}^n \longrightarrow E \longrightarrow B $ over a complex manifold $B$ consists in the construction of a new fibre bundle $P(\xi) :  \mathbb{C}P^{n-1} \longrightarrow P(E) \longrightarrow B, $  called the projectivization bundle of $\xi$, by replacing each fibre $\mathbb{C}^n $ in $\xi$ with the corresponding complex projective space $\mathbb{C}P^{n-1}$. Recall that the projective complex linear group $PGL(n, \mathbb{C})$ is the quotient of the complex general linear group $GL(n, \mathbb{C})$ by the normal subgroup of scalar matrices. Let $\xi :  \mathbb{C}^n \longrightarrow E \longrightarrow B $ be a complex vector bundle with transition functions $g_{\alpha \beta} : U_{\alpha} \cap U_{\beta} \longrightarrow GL(n, \mathbb{C}).$ Consider the composition $$\overline{g_{\alpha \beta}} : U_{\alpha} \cap U_{\beta} \xrightarrow{g_{\alpha \beta}} GL(n, \mathbb{C}) \longrightarrow PGL(n, \mathbb{C}),$$ the projectivization $P(\xi)$ of $\xi$ is the fibre bundle $P(\xi):  \mathbb{C}P^{n-1} \longrightarrow P(E) \longrightarrow B$ with transition functions $\overline{g_{\alpha \beta}}$ \cite{Bott1982}. If the structure group of $\xi$ reduces to the unitary group $\mathrm{U}(n)$, for instance when both $E$ and $B$ are smooth complex manifolds, then the structure group of $P(\xi)$ reduces to the projective unitary group $\mathrm{U}(n)/S^1 = P\mathrm{U}(n)$, where $S^1$ is viewed as a subgroup of $\mathrm{U}(n)$ under the map $\lambda \longrightarrow \text{diag}(\lambda)$. 

Bott and Tu showed that the cohomology ring of the projectivization bundle total space $P(E)$  can be computed using Leray Hirsch theorem and the ring structure $H^*(P(E), \mathbb{Z})$ can be expressed in terms of the Chern classes of the complex vector bundle $\xi$ \cite{Bott1982}. In general, given a complex vector bundle $\xi : \mathbb{C}^n \longrightarrow E \longrightarrow B,$ there exist Chern classes $c_k(\xi) \in H^{2k}(B, \mathbb{Z})$, for $1 \leq k \leq n$ and $c_0(\xi)=1$,  which are homotopic invariants of the bundle \cite{MinlorStasheff1974Characteristic}. The total Chern class $c(\xi)=H^{*}(B, \mathbb{Z})$ is defined by $c(\xi)=1 + c_1(\xi) + \cdots + c_n(\xi)$. Consider the tangent bundle $\tau : \mathbb{C}^n \longrightarrow E \longrightarrow \mathbb{C}P^n$ over the complex projective spaces $\mathbb{C}P^n$. Recall that $ H^{*}(\mathbb{C}P^n,  \mathbb{Z}) \cong \mathbb{Z}[a]/a^{n+1}$, where $c_{1}(\tau) = a $ is a generator of $H^2(\mathbb{C}P^n, \mathbb{Z})$ and the total Chern class of $\tau$ is given by $c(\tau) = (1+a)^{n+1}$ \cite{Bott1982, MinlorStasheff1974Characteristic}. The cohomology algebra of the total space $P(E)$ of the projectized bundle $P(\tau)$  is  
 $$ H^{*}(P(E), \mathbb{Z}) = H^{*}(\mathbb{C}P^{n}, \mathbb{Z})[x] / \left( x^{n} + c_{1}(\tau)x^{n-1} + \cdots + c_{n-1}(\tau)x + c_{n}(\tau) \right),  $$ where $ c_{k}(\tau) \in H^{2k}(\mathbb{C}P^{n}, \mathbb{Z})$ are the Chern classes and $x$ is a  generator of $ H^{2}(\mathbb{C}P^{n-1}, \mathbb{Z})$  \cite{Felix2008, tralle2006symplectic, nishinobu2014lefschetz}. 

Rational homotopy theory presents several results derived from the projectivization bundle $P(\xi)$ and its applications \cite{Bott1982, lambrechts2008poincare, lupton1994symplectic, Thomas1981RHSerreFibratns}. For instance, Lupton and Oprea in \cite{lupton1994symplectic} proved that when $B$ is formal then the total space $P(E)$ is also formal. Thomas in \cite{Thomas1981RHSerreFibratns}  proved that the projectivization bundle $P(\xi)$ is a pure fibration. However, there is lack of comprehensive knowledge about the rational homotopy type of the total space $P(E)$ of the projectivization bundle $P(\xi)$ over the complex projective spaces $\mathbb{C}P^n$.

In this paper, we show the following result:  

\textbf{Theorem 3.2.} \textit{Let $\tau : \mathbb{C}^n \longrightarrow E \longrightarrow \mathbb{C}P^n$ be the tangent bundle over $\mathbb{C}P^n$ where $ n\geq 2$. The total space $P(E)$ of the projectivization bundle $P(\tau) : \mathbb{C}P^{n-1} \rightarrow P(E) \rightarrow \mathbb{C}P^{n}$ has the rational homotopy type of $\mathrm{U}(n+1) / \mathrm{U}(1) \times \mathrm{U}(1) \times \mathrm{U}(n-1).$}

The paper is organized as follows. In \textsection 2 we present some foundational concepts in rational homotopy theory using the Sullivan minimal models. In \textsection 3 we provide the proof of our main result.

\section{Sullivan models}
Here we give basic definitions of Sullivan minimal models theory for which the standard reference is \cite{Felix2001}. 

Let $\displaystyle A = \oplus_{n\geq 0}A^n $  be a graded algebra. It is called commutative if $a\cdot b = (-1)^{n\cdot m}b \cdot a$, where $a \in A^n$ and $b \in A^m$. A differential graded algebra (dga) is a graded algebra $A$ together with a differential which is a linear map $d : A^n \longrightarrow A^{n+1}$ such that $ d(a \cdot b) = (da)\cdot b + (-1)^{n\cdot m} a \cdot (db) $ and $d\circ d = 0.$ Furthermore, a differential graded algebra $A$ satisfying the commutativity property is called a commutative differential graded algebra (cdga). Let $V=\{V^q\}_{q\geq 1}$ be a graded vector space,   the free commutative graded algebra $\Lambda V$ over $V$ is defined as $\Lambda V = S(V^{even}) \otimes E(V^{odd})$, where $E(V^{odd})$ is the exterior algebra and $S(V^{even})$ is the symmetric algebra. If the set \( \{ v_{1}, v_{2}, \cdots \} \) forms a basis for \( V \), then \( \Lambda V \) is commonly denoted as \( \Lambda (v_{1}, v_{2}, \cdots ) \). The suspension $(sV)$ of the graded vector space $V=\{V^q\}_{q\geq 1}$ is the graded vector space defined as $(sV)^n = V^{n+1}$ for all $n$.

A Koszul Sullivan extension is a cdga morphism $(A, d) \xlongrightarrow{i} (A \otimes \Lambda V, d)$ where $V$ has a filtration $V(0) \subseteq V(1) \subseteq \cdots \subseteq V$ such that $dV(0) \subseteq A$ and $dV(k) \subseteq A \otimes V(k-1)$. The cdga $(A \otimes \Lambda V, d)$ is called a relative Sullivan algebra. For a morphism $\varphi : (A, d) \longrightarrow (B, d)$ where $H^0(A) = H^0(B) = \mathbb{Q}$ and $H^1(\varphi)$ is injective, then there exists a relative Sullivan algebra $(A \otimes \Lambda V, d)$ and a quasi-isomorphism  $(A \otimes \Lambda V, d) \xrightarrow{~~ \psi ~~} (B, d)$ such that $\varphi = \psi \circ i$ \cite{Felix2001}. A relative Sullivan algebra $(A \otimes \Lambda V, d)$ is called pure if $D(V^{\text{even}}) = 0$ and $D(V^{\text{odd}}) \subseteq A \otimes \Lambda (V^{\text{even}})$. If $A = \mathbb{Q}$, then $(\Lambda V, d)$ is called a Sullivan algebra. Moreover, a Sullivan algebra $(\Lambda V, d)$ is said to be minimal if $dV(k) \subseteq \Lambda^{\geq 2} V(k-1)$ \cite{Felix2001}. A Sullivan algebra $(\Lambda V, d) $ of the form $ (\Lambda Q \otimes \Lambda P, d)$, where $Q=V^{\text{even}}$ and $P=V^{\text{odd}}$ is called pure if $dQ = 0$ and $dP \subseteq \Lambda Q$ \cite{Felix2001}.

Let $(A, d)$ be a cdga such that  $H^0(A, d) = \mathbb{Q}$ then there exists a Sullivan algebra $(\Lambda V, d)$ together with a quasi-isomorphism given by $m : (\Lambda V, d) \longrightarrow (A, d)$. If $H^1(A, d) = 0$, then $(\Lambda V, d)$ can be chosen to be minimal. Sullivan defined a functor $A_{PL}$ from the category of topological spaces to the category of commutative differential graded algebras. Moreover, if $X$ is simply connected and of finite type, a Sullivan model of $X$ is a Sullivan model $(\Lambda V, d)$ of $A_{PL}(X)$. Therefore, $ H(\Lambda V, d) = H^{*}(X, \mathbb{Q})$. Moreover, if $(\Lambda V, d)$ is minimal then $V^n \cong \text{Hom}_{\mathbb{Z}}(\pi_n(X), \mathbb{Q})$ \cite{DennisSullivan1977}. A cdga model of $X$ is a cdga of $(A, d)$ with a quasi-isomorphism $\varphi : (\Lambda V, d) \xlongrightarrow{} (A,d)$, where $(\Lambda V, d)$ is a Sullivan model of $X$. 

A simply connected space $X$ is called formal if there is a quasi-isomorphism $\phi : (\Lambda V, d) \longrightarrow (H^{*}(X), 0)$. Spheres, complex projective spaces,  homogeneous spaces $G/H$ where $G$ and $H$ have the same rank, are examples of formal spaces \cite{Felix2001}. Assume $X$ is a smooth manifold and $\omega$ a 2-form on $X$. If $\omega$ is closed (i.e., $d\omega = 0$) and non-degenerate (i.e., $\omega^n \neq 0$), then the pair $(X, \omega)$ is called a symplectic manifold \cite{Felix2008}.  

To each fibration \(F \longrightarrow E \xlongrightarrow{p} B \) between simply connected spaces of finite type is associated a Koszul Sullivan extension $\displaystyle (A, d) \xlongrightarrow{i} (A \otimes \Lambda V, D) \longrightarrow (\Lambda V, d)$, which models  the fibration $p$. In particular, $i : (A, d) \xlongrightarrow{} (A \otimes \Lambda V, D)$ is a model of $p$ and $(\Lambda V, d)$ is a Sullivan model of $F$. Moreover, $(A \otimes \Lambda V, D)$ is the cdga model of $E$ \cite{FelixHalperin2023}. A fibration $p$  is called pure if it admits a relative Sullivan model $(A, d_A) \longrightarrow (A \otimes \Lambda V, D) \longrightarrow (\Lambda V, d_V),$ such that $D(V^{\text{even}}) = 0$ and $D(V^{\text{odd}}) \subseteq A \otimes \Lambda (V^{\text{even}})$, where $(A, d_A)$ is a cdga model of $B$ \cite{Thomas1981RHSerreFibratns}. 

\begin{thm}{\cite[Section~15]{Felix2001}, \cite[p.~83]{Felix2008}}
 Let $H$ be a closed connected subgroup of a compact connected Lie group $G$. We denote by $i : H \longrightarrow G$ the canonical inclusion
and consider the classifying map $Bi : BH \longrightarrow BG$. Let $H^{*}(BG; \mathbb{Q}) = \Lambda V$ and $H^{*}(BH; \mathbb{Q}) = \Lambda W$ be respective cohomology algebras of $BG$ and $BH$. Define a differential $d$ on $\Lambda W \otimes \Lambda (sV)$ as
$dw = 0$ if $w \in W$ and $d(sv) = H^{*}(Bi)(v)$ if $v \in V$. The cdga $(\Lambda W \otimes \Lambda (sV), d)$ is a Sullivan model for the homogeneous space $G/H$. In particular, $H^{*}(G/H; \mathbb{Q}) = H^*(\Lambda W \otimes \Lambda (sV), d)$.
\end{thm}
The resulting Sullivan algebra $(\Lambda W \otimes \Lambda sV , d)$ in \textit{Theorem 2.1.} is not necessarily minimal. 

\section{Projectivization of smooth complex vector bundles} 
In this section, we will provide a Sullivan model of the projectivization of a smooth complex vector bundle. 

\begin{prop}
Let $\xi : \mathbb{C}^n \longrightarrow E \longrightarrow B$ be a smooth complex vector bundle, where $B$ is simply connected and $(A, d)$ be the cdga of $B$. Then a Sullivan model of the total space $P(E)$ of the projectivized bundle $P(\xi) : \mathbb{C}P^{n-1} \longrightarrow P(E) \xlongrightarrow{} B$  is given by $$\psi : (A ,  d) \rightarrow ( A \otimes \Lambda (x_2, x_{2n-1}), 
 D), \; Dx_2 = 0 \; \text{and} \; Dx_{2n-1} = x_{2}^{n}+\sum_{i=1}^{n}c_{i}x_{2}^{n-i} ,$$ where $c_i$ is a cocycle in $A^{2i}$ which represents the $i$-th Chern class of $\xi$. 
\end{prop}

\begin{proof}
Applying \textit{Theorem 2.1.}, we get that $P\mathrm{U}(n) = \mathrm{U}(n)/\mathrm{U}(1)$ has a Sullivan model of the form $(\Lambda( \mathrm{y}_{3}, \cdots , \mathrm{y}_{2n-1}), \;0)$. 
Hence, $P\mathrm{U}(n)$ has the rational homotopy type of 
 $S\mathrm{U}(n)$. The cohomology algebra of $BP\mathrm{U}(n)$ is then \( \Lambda(\mathrm{y}_4, \mathrm{y}_6, \ldots, \mathrm{y}_{2n}) \). 
Additionally, the projectivization fibre bundle \[ P(\xi)  :  \mathbb{C}P^{n-1} \longrightarrow P(E) \longrightarrow B, \]
is classified by a map $f: B \rightarrow BP\mathrm{U}(n).$ 
If $(A, d)$ is a cdga model of $B$, then a cdga model of $f$ is  $$\phi : (\Lambda(\mathrm{y}_{4}, \cdots , \mathrm{y}_{2n}), 0) \longrightarrow (A , d),\;\; \phi(\mathrm{y}_{2i})=c_i,$$ where $[c_{i}] \in H^{2i}(A , d)$ are the Chern classes of $\xi$, for $i=1,\cdots , n $. As $(\Lambda(x_{2}, x_{2n-1}),d)$ with $dx_{2}=0, \; dx_{2n-1}=x_{2}^{n}$ is a Sullivan model of $\mathbb{C}P^{n-1}$, then a relative Sullivan model of the projectivization bundle is given by 
 \[ (A, d_A) \xlongrightarrow{\psi} (A \otimes \Lambda(x_{2}, x_{2n-1}), D) \longrightarrow (\Lambda(x_{2}, x_{2n-1}), d), \] with $\displaystyle Dx_{2} = 0, \;\;  Dx_{2n-1} = x_{2}^{n}+\sum_{i=1}^{n}c_{i}x_{2}^{n-i}\;\; \text{and} \;\; D_{|A}=d_A.$ 
 \end{proof}
 
 \begin{cor}
     Let $\tau : \mathbb{C}^n \longrightarrow E \longrightarrow \mathbb{C}P^n$ be the tangent bundle over $\mathbb{C}P^n$ and $P(\tau) : \mathbb{C}P^{n-1} \longrightarrow P(E) \longrightarrow \mathbb{C}P^n$ its projectivization bundle. Then the cohomology algebra of the total space $P(E)$ is given by $H^*(P(E), \mathbb{Q}) = \Lambda (x_2, y_2)/ I$, where $I$ is the ideal generated by \[\Bigg\{ x_{2}^{n}+\sum_{i=1}^{n}y_2^{i}x_{2}^{n-i} \; , \; y_2^{n+1} \Bigg\}.\] 
 \end{cor}
 
\begin{proof}
 The Sullivan minimal model of $\mathbb{C}P^n$ is given by $(\Lambda(y_2, y_{2n+1}), \;d)$, where $dy_2 = 0$ and $dy_{2n+1} = y_2^{n+1}.$ As the total Chern class of the tangent bundle is $c(\tau) = (1+y_2)^{n+1}$, where $y_2$ is a generator of $ H^{2}(\mathbb{C}P^n , \mathbb{Z})$, then $c_i(\tau) = \binom{n+1}{i}y_2^i$ for $i=1,\cdots,n.$ Since $\binom{n+1}{i}y_2^i\neq 0,$ the associated principle $\mathrm{U}(n)$-bundle $\mathrm{U}(n) \longrightarrow P \longrightarrow \mathbb{C}P^n$ is classified by \[\phi : \Lambda(z_2, z_4, \cdots, z_{2n}) \longrightarrow \Lambda(y_2, y_{2n+1}),\] where $\phi(z_{2i})=y_2^i$, after a suitable change of variables. Therefore, a Sullivan model of the projectivized bundle total space $P(E)$ is 
\[( \Lambda(y_{2}, y_{2n+1}) \otimes \Lambda( x_{2},  x_{2n-1}) , \;D),\]
where \[Dx_{2} \;=\; Dy_{2}\;=\;0 ,\;\;
Dy_{2n+1}\;=\; y_{2}^{n+1} \;\;\text{and}\;\; Dx_{2n-1}\;=\;  \sum_{i=0}^{n}x_{2}^{n-i}y_{2}^{i}. \]

As the Sullivan model \((\Lambda(x_2, y_2, x_{2n-1}, y_{2n+1}), D)\) of \(P(E)\) is both pure and symplectic, it implies that \(P(E)\) inherits a symplectic structure from both the fibre \(\mathbb{C}P^{n-1}\) and the base \(\mathbb{C}P^n\) \cite{mcduff2017introduction}. Consequently, \(P(E)\) is formal, as stated in Corollary 2.3 of \cite{lupton1994symplectic}. Hence, the cohomology algebra of $P(E)$ is $H^*(P(E), \mathbb{Q}) \cong \Lambda(x_2,\;y_2)/I$,  where the ideal $I$ is generated by \[ \Bigg\{ x_{2}^{n}+\sum_{i=1}^{n}y_2^{i}x_{2}^{n-i} \; , \; y_2^{n+1} \Bigg\}. \]
\end{proof}

\section{Main results}
In this section we present our main result. First we compute the minimal Sullivan model of the homogeneous space $ \mathrm{U}(n+1)/\mathrm{U}(1) \times \mathrm{U}(1) \times \mathrm{U}(n-1)$ and then determine the rational homotopy type of the total space $P(E)$ of the projectivization bundle $P(\tau)$ of the tangent bundle $\tau$ over $\mathbb{C}P^n$. 

\subsection{Sullivan minimal model of $ \mathrm{U}(n+1)/\mathrm{U}(1) \times \mathrm{U}(1) \times \mathrm{U}(n-1)$}
Here we state and prove the theorem of our first main result.
\begin{thm}
   For $n \geq 2,$ the homogeneous space $ \mathrm{U}(n+1)/\mathrm{U}(1) \times \mathrm{U}(1) \times \mathrm{U}(n-1)$ has a Sullivan model of the form $(\Lambda (a_2, b_2, v_{2n-1}, v_{2n+1}),\; d)$,  where $da_2 \;=\;  db_2 \;=\; 0$,  $\displaystyle dv_{2n-1}= (-1)^{n+1} \sum_{i=0}^{n}a_{2}^{n-i}b_{2}^{i}$ and $\displaystyle dv_{2n+1}=(-1)^{n+1} \sum_{i=1}^{n}a_{2}^{n-i+1}b_{2}^{i}$. In particular, its cohomology algebra is given by 
 $\Lambda(a_{2}, b_{2})/I,$
where $I$ is the ideal generated by  \[  \left\{ (-1)^{n+1} \sum_{i=0}^{n}a_{2}^{n-i}b_{2}^{i} \; , \; (-1)^{n+1} \sum_{i=1}^{n}a_{2}^{n-i+1}b_{2}^{i} \right\}. \]
\end{thm}

\begin{proof}
We apply \textit{Theorem 2.1.} to compute a Sullivan model of  $\mathrm{U}(n+1)/\mathrm{U}(1) \times \mathrm{U}(1) \times \mathrm{U}(n-1)$, for $n \geq 2$. Let $i : H = \mathrm{U}(1) \times \mathrm{U}(1) \times \mathrm{U}(n-1) \hookrightarrow G = \mathrm{U}(n+1)$ be the inclusion. Then applying results in  \cite{Greub-Halp-V3-1976}, the classifying map $Bi : BH \rightarrow BG$ has a Sullivan model $$ \varphi : (\Lambda(v_{2}, v_{4}, v_{6}, \cdots , v_{2n+2}), 0) \rightarrow (\Lambda(a_{2}, b_{2}, z_{2}, z_{4}, z_{6}, \cdots, z_{2n-2}), 0)$$ with 
\begin{eqnarray*}
\varphi(v_{2})&=&a_{2}+b_{2}+z_{2}, \\
\varphi(v_{4}) &=& a_{2}b_{2}+b_{2}z_{2}+a_{2}z_{2}+z_{4}, \\
\varphi(v_{6})&=&a_{2}b_{2}z_{2}+ (a_{2}+b_{2})z_{4} + z_{6}, \\ 
\varphi(v_{8}) &=& a_{2}b_{2}z_{4}+ (a_{2}+b_{2})z_{6} + z_{8}, \\
\vdots &=& \quad \vdots \\
\varphi(v_{2n-2}) \; &=& \; a_{2}b_{2}z_{2n-6} + (a_{2}+b_{2})z_{2n-4} + z_{2n-2},   \\
\varphi(v_{2n})&=&a_{2}b_{2}z_{2n-4} + (a_{2}+b_{2})z_{2n-2},\\
\varphi(v_{2n+2})&=&a_{2}b_{2}z_{2n-2}. 
\end{eqnarray*}

Therefore, a Sullivan model of $\mathrm{U}(n+1) / \mathrm{U}(1) \times \mathrm{U}(1) \times \mathrm{U}(n-1) $ is \[(\Lambda( a_{2}, b_{2}, z_{2},z_{4},\cdots , z_{2n-2}, v_{1}, v_{3}, v_{5},\cdots , v_{2n+1}), d), \;\;\text{where} \] 
\begin{eqnarray*}
da_{2}&=&db_{2}=dz_{2}=0, \\
dv_{1}&=& a_{2}+b_{2}+z_{2}, \\
dv_{3}&=& a_{2}b_{2}+b_{2}z_{2}+a_{2}z_{2}+z_{4}, \\
dv_{5}&=&a_{2}b_{2}z_{2}+ (a_{2}+b_{2})z_{4} + z_{6}, \\
dv_{7}&=&a_{2}b_{2}z_{4}+ (a_{2}+b_{2})z_{6} + z_{8}, \\
\vdots~~~ &=& \quad \vdots  \\
dv_{2n-3}&=&a_{2}b_{2}z_{2n-6} + (a_{2}+b_{2})z_{2n-4} + z_{2n-2},  \\ 
dv_{2n-1}&=&a_{2}b_{2}z_{2n-4} + (a_{2}+b_{2})z_{2n-2}, \\
dv_{2n+1}&=&a_{2}b_{2}z_{2n-2}. 
\end{eqnarray*}

Since $dv_{1}$ is linear, we let $t_{2} = a_{2}+b_{2}+z_{2}$ and by the substitution of $t_{2}$ we obtain the following isomorphic Sullivan algebra \[(\Lambda( a_{2}, b_{2}, t_{2},z_{4},\cdots , z_{2n-2}, v_{1}, v_{3}, v_{5},\cdots , v_{2n+1}), d), \;\;\text{where}\]
\begin{eqnarray*}
da_{2}&=&db_{2}=dz_{2}=0, \\
dv_{1}&=&t_{2}, \\
dv_{3}& =& a_{2}b_{2}+(a_{2}+b_{2})[t_{2}-(a_{2}+b_{2})]+z_{4}, \\
dv_{5}&=&a_{2}b_{2}[t_{2}-(a_{2}+b_{2})]+ (a_{2}+b_{2})z_{4} + z_{6}, \\
dv_{7}& =& a_{2}b_{2}z_{4}+ (a_{2}+b_{2})z_{6} + z_{8} \; , \\
\vdots~~~ &=& \quad \vdots  \\
dv_{2n-3} &=& a_{2}b_{2}z_{2n-6} + (a_{2}+b_{2})z_{2n-4} + z_{2n-2},  \\
dv_{2n-1}&=&a_{2}b_{2}z_{2n-4} + (a_{2}+b_{2})z_{2n-2}, \\
dv_{2n+1} &=& a_{2}b_{2}z_{2n-2}.
\end{eqnarray*}
Therefore, the cancellation by the acyclic ideal generated by $v_{1}$ and $t_{2}$ gives a Sullivan algebra
 \[(\Lambda( a_{2}, b_{2},z_{4},z_{6},\cdots , z_{2n-2}, v_{3}, v_{5},\cdots , v_{2n+1}), d), \;\;\text{where}\]
\begin{eqnarray*}
da_{2}&=&db_{2}\;=\;0, \\
dv_{3}&=&a_{2}b_{2}-(a_{2}+b_{2})^2 +z_{4}, \\
dv_{5}&=&\left(-a_{2}b_{2}+z_4\right)(a_{2}+b_{2}) + z_{6}, \\
dv_{7}&=&a_{2}b_{2}z_{4}+ (a_{2}+b_{2})z_{6} + z_{8}, \\
\vdots~~~ &=& \quad \vdots  \\
dv_{2n-3}&=&a_{2}b_{2}z_{2n-6} + (a_{2}+b_{2})z_{2n-4} + z_{2n-2},  \\
dv_{2n-1}&=&a_{2}b_{2}z_{2n-4} + (a_{2}+b_{2})z_{2n-2}, \\
dv_{2n+1} &=& a_{2}b_{2}z_{2n-2}.
\end{eqnarray*}
 
Letting  $t_{4} = a_{2}b_{2}-(a_{2}+b_{2})^{2}+z_{4}$ yields the following  Sullivan algebra  \[(\Lambda( a_{2}, b_{2},t_{4},z_{6},\cdots , z_{2n-2}, v_{1}, v_{3}, v_{5},\cdots , v_{2n+1}), d), \;\;\text{where}\]
\begin{eqnarray*}
da_{2}&=&db_{2}\;=\;0, \\
dv_{3}& = &t_{4}, \\
dv_{5}&=&  (a_{2}+b_{2})  \Bigr[-a_{2}b_{2} + t_4 -a_{2}b_{2} + (a_{2}+b_{2})^2 \Bigr] + z_{6}, \\
dv_{7}&=&a_{2}b_{2}\Bigr[ t_4 -a_{2}b_{2} + (a_{2}+b_{2})^2 \Bigr]+ (a_{2}+b_{2})z_{6} + z_{8}, \\
\vdots~~~ &=& \quad \vdots  \\
dv_{2n-3}&=&a_{2}b_{2}z_{2n-6} + (a_{2}+b_{2})z_{2n-4} + z_{2n-2},   \\
dv_{2n-1} &=& a_{2}b_{2}z_{2n-4} + (a_{2}+b_{2})z_{2n-2}, \\
dv_{2n+1}&=&a_{2}b_{2}z_{2n-2}.
\end{eqnarray*}

Canceling by the acyclic ideal generated by $v_{3}$ and $t_{4}$, we obtain a Sullivan algebra
 \[(\Lambda( a_{2}, b_{2},z_{6},\cdots , z_{2n-2}, v_{5},\cdots , v_{2n+1}), d),\] where
\begin{eqnarray*}
da_{2}&=&db_{2}\;=\;0, \\
dv_{5}& = &-2a_{2}b_{2}(a_{2}+b_{2})+ (a_{2}+b_{2})^{3} + z_{6}, \\
dv_{7}&=&-(a_{2}b_{2})^2 + a_{2}b_{2}(a_{2}+b_{2})^2 + (a_{2}+b_{2})z_{6} + z_{8}, \\
\vdots~~~ &=& \quad \vdots  \\
dv_{2n-3}&=&a_{2}b_{2}z_{2n-6} + (a_{2}+b_{2})z_{2n-4} + z_{2n-2},   \\
dv_{2n-1}&=& a_{2}b_{2}z_{2n-4} + (a_{2}+b_{2})z_{2n-2}, \\
dv_{2n+1}&=&a_{2}b_{2}z_{2n-2}.
\end{eqnarray*}

The process of changing variables continues up to $t_{2n-2} = a_{2}b_{2}z_{2n-6} + (a_{2}+b_{2})z_{2n-4} + z_{2n-2}$, yielding the Sullivan algebra  $$(\Lambda( a_{2}, b_{2}, t_{2n-2}, v_{2n-3}, v_{2n-1}, v_{2n+1}), \;d),$$ where
\begin{eqnarray*}
da_{2}&=&db_{2}\;=\;0, \\
dv_{2n-3}&=& t_{2n-2},   \\
dv_{2n-1}&=&a_{2}b_{2}z_{2n-4} + (a_{2}+b_{2})\Bigr[ t_{2n-2} - a_{2}b_{2}z_{2n-6} -  (a_{2}+b_{2})z_{2n-4} \Bigr],\\
dv_{2n+1}&=&a_{2}b_{2}\Bigr[ t_{2n-2} - a_{2}b_{2}z_{2n-6} -  (a_{2}+b_{2})z_{2n-4} \Bigr].
\end{eqnarray*}
By cancelling the acyclic ideal generated by $\{t_{2n-2} , v_{2n-3}\}$, we get a Sullivan minimal model 
$$(\Lambda( a_{2}, b_{2},v_{2n-1}, v_{2n+1}), d),$$ where
\[da_{2} \;=\; db_{2}\;=\;0 ,\;\;
dv_{2n-1}\;=\; (-1)^{n+1} \sum_{i=0}^{n}a_{2}^{n-i}b_{2}^{i} \;\;\text{and}\;\;
dv_{2n+1}\;=\; (-1)^{n+1} \sum_{i=1}^{n}a_{2}^{n-i+1}b_{2}^{i}. \]

 Hence, the cohomology algebra of $\mathrm{U}(n+1)/\mathrm{U}(1) \times \mathrm{U}(1) \times \mathrm{U}(n-1)$ is
\[H^{*}(\mathrm{U}(n+1)/\mathrm{U}(1) \times \mathrm{U}(1) \times \mathrm{U}(n-1), \mathbb{Q}) = \Lambda(a_{2}, b_{2})/ I, \] where $I$ is generated by \[ \left\{ (-1)^{n+1} \sum_{i=0}^{n}a_{2}^{n-i}b_{2}^{i}, \;\;
 (-1)^{n+1} \sum_{i=1}^{n}a_{2}^{n-i+1}b_{2}^{i}\right\}. \]
\end{proof}

\subsection{Rational homotopy type of the projectivization of the tangent bundle over $\mathbb{C}P^{n}$}
In this section, we determine the rational homotopy type of the total space $P(E)$ of the projectivization of the tangent bundle $\tau : \mathbb{C}^n \longrightarrow E \longrightarrow \mathbb{C}P^{n}$.

\begin{thm}
 Given the tangent bundle $\tau : \mathbb{C}^n \longrightarrow E \longrightarrow \mathbb{C}P^n$ over $\mathbb{C}P^n$ where $ n\geq 2$, then the total space $P(E)$ of the projectivization bundle $P(\tau) : \mathbb{C}P^{n-1} \rightarrow P(E) \rightarrow \mathbb{C}P^{n}$ has the rational homotopy type of $\mathrm{U}(n+1) / \mathrm{U}(1) \times \mathrm{U}(1) \times \mathrm{U}(n-1).$
\end{thm}

\begin{proof}
By the previous theorem the Sullivan minimal model of $\mathrm{U}(n+1) / \mathrm{U}(1) \times \mathrm{U}(1) \times \mathrm{U}(n-1)$ is given  by 
\[ (\Lambda( a_{2}, b_{2},v_{2n-1}, v_{2n+1}), d),\] where
\[da_{2} \;=\; db_{2}\;=\;0 ,\;\;
dv_{2n-1}\;=\; (-1)^{n+1} \sum_{i=0}^{n}a_{2}^{n-i}b_{2}^{i} \;\;\text{and}\;\;
dv_{2n+1}\;=\; (-1)^{n+1} \sum_{i=1}^{n}a_{2}^{n-i+1}b_{2}^{i}. \]

Moreover, the Sullivan model of the projectivized bundle $P(E)$ is  given in \textit{Corollary 3.2.} as 
\[( \Lambda(y_{2}, y_{2n+1}) \otimes \Lambda( x_{2},  x_{2n-1}) , d),\]
where \[dx_{2} \;=\; dy_{2}\;=\;0 ,\;\;
dx_{2n-1}\;=\;  \sum_{i=0}^{n}x_{2}^{n-i}y_{2}^{i} \;\;\text{and}\;\;
dy_{2n+1}\;=\; y_{2}^{n+1}. \]

We define a morphism $$f : (\Lambda( a_{2}, b_{2},v_{2n-1}, v_{2n+1}), d) \rightarrow (\Lambda( x_{2}, y_{2}, x_{2n-1}, y_{2n+1}), d),$$
 by $f({a_{2}})=(-1)^{n+1}x_{2}$, $f({b_{2}})=(-1)^{n+1}y_{2}$, $f({v_{2n-1}})=(-1)^{n+1}x_{2n-1}$ and $f({v_{2n+1}})=(-1)^{n+1}(y_{2}x_{2n-1}-y_{2n+1})$. 
It is easily seen that $f$ commutes with differentials. A filtration by world length gives rise to an isomorphism on the associated bigraded modules. A standard spectral sequence argument yields that $f$ is an isomorphism (see \cite[Prop.~18.2]{Felix2001}). 
Hence,  $P(E)$ has the rational homotopy type of 
 $\mathrm{U}(n+1) / \mathrm{U}(1) \times \mathrm{U}(1) \times \mathrm{U}(n-1).$  
 \end{proof}

\textbf{Acknowledgement}

This study is based on a portion of Ndlovu's Ph.D. thesis, which was conducted under Gatsinzi's supervision.

\end{document}